\documentclass{elsarticle}

\makeatletter
\def\ps@pprintTitle{%
	\let\@oddhead\@empty
	\let\@evenhead\@empty
	\def\@oddfoot{}%
	\let\@evenfoot\@oddfoot}
\makeatother

\usepackage{amssymb, amsthm}
\usepackage{amsmath} 				% A mathematical package, providing elements like ``bmatrix''
\usepackage{amsfonts}
\usepackage{bbm,alltt}
\usepackage{tabularx}
\usepackage{multirow}
\usepackage{float} 
\usepackage{graphicx}      % Required to insert images
\usepackage[caption=false]{subfig}  % edmond 6/16/2017
\usepackage{blkarray}

\usepackage{algorithm}
\usepackage{algorithmic} % Algorithm style, could alternatively use algpseudocode

\newtheorem{theorem}{Theorem}

\newtheorem{proposition}{Proposition}
	
\usepackage[colorlinks]{hyperref}
\usepackage[capitalise,noabbrev]{cleveref}	% for \cref
\crefformat{equation}{(#2#1#3)}
\usepackage{lipsum}
\usepackage{amsfonts}
\usepackage{graphicx}
\usepackage{booktabs} 				% Horizontal rules in tables
\usepackage{verbatim} 				% edmond 6/12/2017
\usepackage{epstopdf}
\ifpdf%
  \DeclareGraphicsExtensions{.eps,.pdf,.png,.jpg}
\else
  \DeclareGraphicsExtensions{.eps}
\fi
\usepackage{amsopn}

\begin{document}
	
\begin{frontmatter}
	
\title{Error analysis of an accelerated interpolative decomposition for 3D Laplace 
	problems}

\author[affi]{Xin Xing}
\ead{xxing33@gatech.edu}

\author[affi]{Edmond Chow}
\ead{echow@cc.gatech.edu}

\address[affi]{School of Computational Science and Engineering, 
	Georgia Institute of Technology, Atlanta, GA}
	
\begin{abstract}
	In constructing the $\mathcal{H}^2$ representation of dense matrices defined by the Laplace kernel,
the interpolative decomposition of certain off-diagonal submatrices that dominates the computation
can be dramatically accelerated using the concept of a proxy surface. 
We refer to the computation of such interpolative decompositions as the proxy surface method. 
We present an error bound for the proxy surface method in the 3D case and thus provide theoretical  
guidance for the discretization of the proxy surface in the method. 
%We present a detailed error analysis of this method for the 3D Laplace kernel to show the accuracy of the method and 
%also to provide theoretical guidance for the discretization of the proxy surface in the method. 

\end{abstract}
	
\end{frontmatter}

%\title{Error analysis of an accelerated interpolative decomposition for 3D Laplace 
%	problems
%	\thanks{Version of \today.}}
%
%\author{
%	Xin Xing%
%	\thanks{School of Computational Science and Engineering, Georgia Institute of Technology,
%		Atlanta, GA (\email{xxing33@gatech.edu}, \email{echow@cc.gatech.edu}).}
%	\and
%	Edmond Chow%
%	\footnotemark[2]
%}
%
%\headers{Error analysis of the proxy surface method}
%{Xin Xing and Edmond Chow}

%\begin{document}

%\maketitle

\begin{abstract}
In constructing the $\mathcal{H}^2$ representation of dense matrices defined by the Laplace kernel,
the interpolative decomposition of certain off-diagonal submatrices that dominates the computation
can be dramatically accelerated using the concept of a proxy surface. 
We refer to the computation of such interpolative decompositions as the proxy surface method. 
We present an error bound for the proxy surface method in the 3D case and thus provide theoretical  
guidance for the discretization of the proxy surface in the method. 
%We present a detailed error analysis of this method for the 3D Laplace kernel to show the accuracy of the method and 
%also to provide theoretical guidance for the discretization of the proxy surface in the method. 

\end{abstract}

%\keywords{interpolative decomposition \sep proxy surface \sep Laplace kernel}

%\linenumbers
\section{Introduction}
$\mathcal{H}^2$ matrix techniques 
\cite{chandrasekaran_fast_2006, hackbusch_data-sparse_2002,hackbusch_$mathcalh^2$-matrices_2000}
can accelerate dense matrix-vector multiplications and also provide efficient direct solvers for many types 
of dense kernel matrices arising from the discretization of integral equations. 
However, these benefits are based on a rather expensive $\mathcal{H}^2$ matrix construction cost.
Given a kernel function $K(x,y)$, the main bottleneck of $\mathcal{H}^2$ 
construction using interpolative decomposition (ID) 
\cite{cai_difeng_smash:_2017, ho_fast_2012, martinsson_fast_2011} is the ID approximation
of certain kernel submatrices of the form $K(X_0, Y_0)$ where point set $X_0$ lies in a bounded domain  
$\mathcal{X}$ and point set $Y_0$ lies in the far field of $\mathcal{X}$, denoted as $\mathcal{Y}$,
with $Y_0$ usually much larger than $X_0$. A 2D example of $X_0$ and $Y_0$ is shown in \cref{fig:proxy_surface}. 

\begin{figure}[ht]
	\centering
	\includegraphics[width=0.75\textwidth]{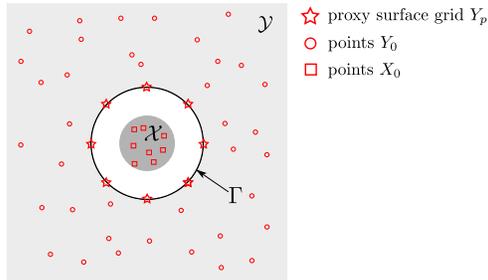}
	\caption{Illustration of the proxy surface method. The matrix to be directly compressed is $K(X_0, Y_p)$,
		rather than $K(X_0, Y_0)$, with a fixed number of columns $|Y_p|$, 
		regardless of how many points $Y_0$ there are in $\mathcal{Y}$.\label{fig:proxy_surface}}
\end{figure}

An algebraic approach to obtain these IDs usually leads to a prohibitive quadratic $\mathcal{H}^2$ construction cost. 
For the Laplace kernel, Martinsson and Rohklin \cite{martinsson_fast_2005} efficiently obtained an 
ID of $K (X_0, Y_0)$ by using the concept of a proxy surface and this 
concept is also used in recursive skeletonization by Ho and Greengard \cite{ho_fast_2012}.
The key idea, as illustrated in \cref{fig:proxy_surface}, is to convert the problem into the ID approximation of a kernel matrix 
$K(X_0, Y_p)$, where point set $Y_p$ is selected to discretize the interior boundary of 
$\mathcal{Y}$, with $Y_p$ much smaller than $Y_0$ in practice. 
The interior boundary, denoted as $\Gamma$, is called a proxy surface in 
\cite{ho_fast_2012} and thus we refer to the method as the \textit{proxy surface method}.
With this method, the $\mathcal{H}^2$ construction cost can be reduced to linear complexity.
It is worth noting that kernel independent FMM \cite{ying_kernel-independent_2004} and the 
proxy point method \cite{xin_proxy_2018} are also based on similar ideas. 

The error analysis of the proxy surface method, however, is only briefly discussed in 
\cite{martinsson_fast_2005} without much detail and the selection of $Y_p$ 
to discretize the proxy surface is heuristic in previous applications 
\cite{martinsson_fast_2005, ho_fast_2012, kong_adaptive_2011, corona_on_2015}. 
In this paper, we provide a detailed error analysis of the proxy surface method 
for the 3D Laplace kernel. 
The error analysis shows that, under certain conditions, it is sufficient to discretize proxy surfaces of 
different sizes using a constant number of points while maintaining a fixed accuracy in the method.

\section{Background}
Given a matrix $A\in \mathbb{R}^{n\times m}$, a rank-$k$ interpolative decomposition (ID) 
\cite{cheng_compression_2005, gu_efficient_1996} of $A$ is of the form
$U A_J$ where $A_J \in \mathbb{R}^{k\times m}$ is a row subset of $A$ 
and $U \in \mathbb{R}^{n\times k}$ has bounded entries. We call $A_J$ and $U$ the skeleton
and projection matrices, respectively. The ID is said to have
precision $\varepsilon$ if the norm of each row of the error matrix $A - U A_J$ is bounded 
by $\varepsilon$. Using an algebraic approach, the ID can be calculated based on the strong rank-revealing 
QR (sRRQR) \cite{gu_efficient_1996} of $A^T$ with entries of the obtained $U$ bounded by
 a pre-specified parameter $C \geqslant 1$.

Take the domain pair $\mathcal{X} \times \mathcal{Y}$ and the interior boundary $\Gamma$ of 
$\mathcal{Y}$  shown in \cref{fig:proxy_surface} as an example. 
For the Laplace kernel $K(x,y)$ and any point sets $X_0 \subset \mathcal{X}$ and $Y_0 \subset \mathcal{Y}$,
we now explain the proxy surface method for the ID approximation of $K(X_0,Y_0)$, based on the discussion 
from \cite{martinsson_fast_2005}.

By Green's Theorem, the potential at any $x \in \mathcal{X}$ generated by source point set $Y_0$ with charges
$F = (f_i)_{y_i \in Y_0}$ can also be generated by an equivalent charge distribution on the proxy surface 
$\Gamma$ that encloses $\mathcal{X}$. Select a point set $Y_p$ uniformly distributed on $\Gamma$ to 
discretize the equivalent charge distribution with point charges $\tilde{F} = (\tilde{f_i})_{y_i \in Y_p}$ at 
$Y_p$. It is shown in \cite{martinsson_fast_2005} that 
\begin{equation*}
\tilde{F} \approx W_{Y_0,Y_p}F,
\end{equation*}
where $W_{Y_0,Y_p}$ is a discrete approximation of the linear operator that maps charges $F$ at $Y_0$ to an 
equivalent charge distribution on $\Gamma$ with $\|W_{Y_0,Y_p}\|_2$ bounded as a consequence of Green's Theorem. 
Matching the potentials induced by $F$ and by $\tilde{F}$ at any $x \in \mathcal{X}$ gives 
$K(x, Y_0)F \approx K(x, Y_p) W_{Y_0, Y_p}F$ and thus it holds that 
\begin{equation}\label{eqn:equi_potential}
K(x, Y_0) \approx K(x, Y_p) W_{Y_0, Y_p}, \quad \forall x \in \mathcal{X},
\end{equation}
where $K(x, Y_0) = (K(x, y_i))_{y_i \in Y_0}$ and $K(x, Y_p) = (K(x, y_i))_{y_i \in Y_p}$ are both row vectors. 
Substituting $X_0 \subset \mathcal{X}$ into \cref{eqn:equi_potential}, the target matrix $K(X_0, Y_0)$ 
can be approximated as $K(X_0, Y_p) W_{Y_0,Y_p}$. Find an ID of $K(X_0, Y_p)$ using sRRQR as 
\begin{equation}\label{eqn:proxy}
K(X_0, Y_p) \approx UK(X_\text{rep},Y_p),
\end{equation}
where $X_\text{rep}\subset X_0$ denotes the ``representative'' point subset associated
with the selected row subset in the skeleton matrix of the ID with $U$ being the projection 
matrix of the ID. The proxy surface method then defines the ID of $K(X_0, Y_0)$ as 
\begin{equation}\label{eqn:target}
K(X_0, Y_0) \approx U K(X_\text{rep}, Y_0).
\end{equation}
The error of the above approximation can be bounded as
\begin{align}
\|K(X_0, Y_0) - U K(X_\text{rep}, Y_0)\|_F  
 &\approx \|\left(K(X_0, Y_p) - U K(X_\text{rep}, Y_p)\right)W_{Y_0,Y_p}\|_F \nonumber \\
 &\leqslant \|K(X_0, Y_p) - U K(X_\text{rep}, Y_p)\|_F \|W_{Y_0,Y_p}\|_2, \nonumber % \label{eqn:error_bound_old}
\end{align}
and thus the error is controlled by the error of the ID in \cref{eqn:proxy}.

The number of points in $Y_p$ used to discretize $\Gamma$ is usually chosen heuristically. 
Ref.\ \cite{martinsson_fast_2005, kong_adaptive_2011} suggest using $|Y_p|\sim O(|X_0|)$ and \cite{ho_fast_2012} claims correctly 
but without an explanation that for the Laplace kernel, 
proxy surfaces of different sizes can be discretized using a constant number of points and this constant only depends 
on the compression precision. 

In this paper, to theoretically justify the proxy surface method, we address the following two problems: 
(a) the quantitative relationship between the errors of the two IDs in \cref{eqn:proxy} and \cref{eqn:target} 
and (b) how to choose the number of points in $Y_p$ to guarantee the accuracy of the proposed ID in 
\cref{eqn:target}. 

\section{Main result}
We focus on the proxy surface method for the 3D Laplace kernel $K(x,y) = 1/|x-y|$.
Denote the open ball of radius $r$ centered at the origin as $B(0, r)$. 
For conciseness of the error analysis, we consider $\mathcal{X} = B(0, r_1)$,
$\mathcal{Y} = \mathbb{R}^3 \backslash B(0, r_2)$, and $\Gamma = \partial B(0, r_2)$ 
with $r_2 > r_1$, as illustrated in \cref{fig:domain}. 
Assume a point set $Y_p$ has been selected to discretize $\Gamma$ and the target
kernel matrix $K(X_0, Y_0)$ is associated with point sets $X_0\subset \mathcal{X}$ and $Y_0 \subset \mathcal{Y}$. 

\begin{figure}[ht]
	\centering
	\includegraphics[width=0.50\textwidth]{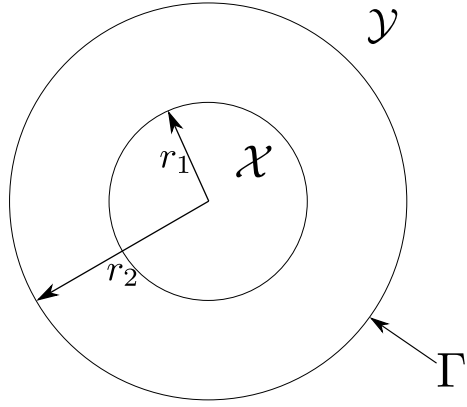}
	\caption{2D illustration of the 3D domain pair $\mathcal{X}\times\mathcal{Y}$ and the proxy 
		surface $\Gamma$.\label{fig:domain}}
\end{figure}

In the proxy surface method, the proposed ID in \cref{eqn:target} can be 
viewed row-by-row as
\[
  K (x_i, Y_0) \approx u_i^T K (X_\text{rep}, Y_0), \quad x_i \in X_0,
\]
where $u_i^T$ denotes the $i$th row of $U$. 
Since the above approximation can be applied to any point set $Y_0$ in $\mathcal{Y}$,
its error is intrinsically based on the function approximation
\begin{equation*}
  K (x_i, y) \approx u_i^T K (X_\text{rep}, y), \quad x_i \in X_0,\ y \in
  \mathcal{Y}.
\end{equation*}
Denote the error of this function approximation by the scalar function
\begin{equation}\label{def:error}
  e_i (y) = K (x_i, y) - u_i^T K (X_\text{rep}, y), \quad x_i \in X_0,\ y \in
  \mathcal{Y}.
\end{equation}
In other words, $e_i(y)$ is the error in the approximation of the interaction 
between $x_i$ and some $y \in \mathcal{Y}$. 
Using this notation, the error of the $i$th row of the approximations \cref{eqn:proxy} and 
\cref{eqn:target} can be denoted as $e_i (Y_0)$ and $e_i (Y_p)$, respectively, which are row vectors. 
In the following discussion, we assume that the ID \cref{eqn:proxy} of $K(X_0, Y_p)$ has precision
$\varepsilon\sqrt{|Y_p|}$ and thus  $\|e_i(Y_p)\|_2 \leqslant \varepsilon\sqrt{|Y_p|}$.

For $Y_0 \subset \mathcal{Y}$ with an arbitrary point distribution, the best upper bound 
for $e_i(Y_0)$ is 
\begin{equation}\label{eqn:maxbound}
  \| e_i (Y_0) \|_2 \leqslant \sqrt{| Y_0 |} \max_{y \in \mathcal{Y}} |e_i (y)| .
\end{equation}
Our error analysis of the proxy surface method seeks an upper bound for 
$|e_i(y)|$ in the whole domain $\mathcal{Y}$ under the condition that 
$\|e_i(Y_p)\|_2 \leqslant \varepsilon \sqrt{|Y_p|} $. In fact, we can prove the following proposition. 
\begin{proposition}\label{prop}
	If point set $Y_p\subset \Gamma$ satisfies the condition that numerical quadrature with 
	the points in $Y_p$	and equal weights $\frac{4\pi r_2^2}{|Y_p|}$ is exact for polynomials
	on $\Gamma$ of degree up to $2c$ where $c$ is an integer constant, then $e_i(y)$ for any $x_i \in X_0$ can be bounded as 
	\begin{equation}\label{eqn:bound}
	| e_i (y) | \leqslant (c + 1) \frac{\| e_i (Y_p) \|_2}{\sqrt{| Y_p |}} +
	(c + 2) \frac{(1 + | X_\textnormal{rep} | \| u_i \|_{\infty})}{r_2 - r_1}
	\left( \frac{r_1}{r_2} \right)^{c + 1}\!\!, \ y \in \mathcal{Y}.
	\end{equation}
%	Thus, the ID approximation \cref{eqn:target} proposed by the proxy surface method
%	has average entry-wise error at each row bounded as
%	\begin{equation}\label{eqn:boundvector}
%	\frac{\|e_i(Y_0)\|}{\sqrt{|Y_0|}} \leqslant (c + 1) \frac{\| e_i (Y_p) \|_2}{\sqrt{| Y_p |}} +
%	(c + 2) \frac{(1 + | X_\text{rep} | \| u_i \|_{\infty})}{r_2 - r_1}
%	\left( \frac{r_1}{r_2} \right)^{c + 1}\!, \ x_i \in X_0.
%	\end{equation}
\end{proposition}

\begin{proof}
For any $x \in \mathcal{X}$, $K(x,y)$ as a function of $y$ is harmonic in $\mathcal{Y}$. 
Since $e_i(y)$ is a linear combination of $K(x_i, y)$ and $\{K(x_j, y): x_j \in X_\text{rep}\}$, 
it is also harmonic in $\mathcal{Y}$. By the maximum principle of harmonic functions, 
$e_i(y)$ satisfies 
\begin{equation}\label{eqn:maximum_principle}
\max_{y \in \mathcal{Y}}  | e_i (y) | = \max_{y \in \Gamma} | e_i (y) | .
\end{equation}
Thus, it suffices to prove the upper bound \cref{eqn:bound} for $y \in \Gamma$. 

The multipole expansion of $K(x,y)$ with $(x, y)\in \mathcal{X}\times\Gamma$ 
is written as	
\begin{equation}\label{eqn:multipole}
  K (x, y) = \sum_{l = 0}^{\infty} \sum_{m = - l}^l M_l^m (x) \frac{1}{r_2^{l + 1}} 
  Y_l^m (\alpha, \beta),
\end{equation}
where $(r_2, \alpha, \beta)$ denotes the polar coordinates of $y$ on $\Gamma$, 
$\{Y_l^m (\alpha, \beta) \}$ is the set of spherical harmonics and 
$\left\{M_l^m(x)\right\}$ is a set of known analytic functions of $x$.
Truncating the above infinite sum at index $c$, the remainder can be bounded 
as
\begin{equation*}
  \left| K (x, y) - \sum_{l = 0}^c \sum_{m = - l}^l M_l^m (x) 
  \frac{1}{r_2^{l + 1}} Y_l^m (\alpha, \beta) \right| 
  \leqslant \frac{1}{r_2 - r_1} \left( \frac{r_1}{r_2} \right)^{c + 1} .
\end{equation*}

Using the above multipole expansion, $e_i(y)$ on $\Gamma$ can be written as 
\begin{align}
  e_i (y) 
  & = \sum_{l = 0}^{\infty} \sum_{m = -l}^l \left( M_l^m (x_i) - u_i^T
  M_l^m (X_\text{rep})\right) \frac{1}{r_2^{l + 1}} Y_l^m (\alpha, \beta) \nonumber \\
  & = \sum_{l = 0}^c \sum_{m = - l}^l E_l^m Y_l^m (\alpha, \beta) + R_c (y),  \label{eqn:multipole_ei}
\end{align}
where $E_l^m$ denotes the coefficient collected for $Y_l^m(\alpha,\beta)$ and
the remainder $R_c(y)$ can be bounded as
\begin{equation*}
  | R_c (y) | \leqslant \frac{(1 + | X_\text{rep} | \| u_i \|_{\infty})}{r_2
  - r_1} \left( \frac{r_1}{r_2} \right)^{c + 1},
\end{equation*}
using the triangle inequality. 
Since $\{ Y_l^m (\alpha, \beta) \}$ is an orthonormal function set on the unit sphere $\mathbb{S}^2$, 
the coefficients $E_l^m$ can be analytically calculated as
\begin{equation*}
  E_l^m = \int_{\mathbb{S}^2} e_i(r_2 y) Y_l^m (y) \mathrm{d} y = \dfrac{1}{r_2^2}\int_{\Gamma} e_i(y) Y_l^m (y) \mathrm{d} y,
\end{equation*}
where $Y_l^m (y)$ is defined as $Y_l^m (\alpha, \beta)$ for any $y = (| y |,
\alpha, \beta)$.

Since numerical quadrature with the points in $Y_p$ and equal weights $\frac{4\pi r_2^2}{|Y_p|}$
is exact for polynomials on $\Gamma$ of degree up to $2c$, $E_l^m$ with $l \leqslant c$ can be 
further represented as
\begin{align}
  E_l^m 
  = & \frac{1}{r_2^2} \int_{\Gamma} (e_i (y) - R_c (y)) Y_l^m (y) \mathrm{d} y   \nonumber\\
  = & \frac{4 \pi}{|Y_p|} \sum_{y_j \in Y_p} (e_i (y_j) - R_c (y_j)) Y_l^m (y_j) \nonumber\\
  = & \frac{4 \pi}{|Y_p|} \left(e_i (Y_p) - R_c (Y_p)\right)^T Y_l^m (Y_p). \label{eqn:coef}
\end{align}
Substituting this $E_l^m$ into \cref{eqn:multipole_ei}, $e_i (y)$ on $\Gamma$ can be written as,

\begin{align}
  e_i (y) 
  = & \frac{4\pi}{|Y_p|} \left(e_i (Y_p)\! - R_c (Y_p)\right)^T\!
  \left(Y_0^0 (Y_p)\ Y_1^{-1}(Y_p)\  \ldots \ Y_c^c(Y_p)\right)\! 
  \left( \begin{array}{c}
    Y_0^0 (y)\\
    Y_1^{- 1} (y)\\
    \vdots\\
    Y_c^c (y)
  \end{array} \right)\! + R_c (y) \nonumber \\
  = & \frac{4 \pi}{|Y_p|} \left(e_i (Y_p) - R_c (Y_p)\right)^T M \Phi (y) + R_c (y), \label{eqn:form_ei}
\end{align}
where $M$ denotes the middle matrix and $\Phi (y)$ denotes the last
vector function of $y$ in the first equation.
Note that any two distinct columns of $M$, say $Y_{l_1}^{m_1} (Y_p)$ and $Y_{l_2}^{m_2}(Y_p)$, are orthogonal with
\begin{equation*}
  Y_{l_1}^{m_1} (Y_p)^T Y_{l_2}^{m_2} (Y_p) 
  = \frac{|Y_p|}{4 \pi} \int_{\mathbb{S}^2} Y_{l_1}^{m_1} (y) Y_{l_2}^{m_2} (y) \mathrm{d} y 
  = \frac{|Y_p|}{4 \pi} \delta_{l_1 = l_2, m_1 = m_2},
\end{equation*}
and thus the scaled matrix $\sqrt{\frac{4 \pi}{|Y_p|}} M$ has orthonormal columns. Therefore, it holds that
\begin{equation*}
\sqrt{\frac{4\pi}{|Y_p|}} \left\| M \Phi(y)\right\|_2 = \|\Phi(y)\|_2.
\end{equation*}
%\begin{align*}
%  \sqrt{\frac{4\pi}{|Y_p|}} \left\| e_i (Y_p)^T M \right\|_2 & \leqslant \| e_i(Y_p) \|_2, \\
%  \sqrt{\frac{4\pi}{|Y_p|}} \left\| R_c (Y_p)^T M \right\|_2 & \leqslant \| R_c (Y_p) \|_2 
%    \leqslant \sqrt{|Y_p|} \frac{(1 + | X_\text{rep} | \| u_i
%  	\|_{\infty})}{r_2 - r_1} \left( \frac{r_1}{r_2} \right)^{c + 1} .
%\end{align*}
Meanwhile, by the property of spherical harmonics, the 2-norm of the
vector function $\Phi (y)$ at any $y \in \Gamma$ is
\begin{equation*}
  \| \Phi (y) \|_2 = \sqrt{\sum_{l = 0}^c \sum_{m = - l}^l | Y_l^m (y) |^2}
  = \sqrt{\sum_{l = 0}^c \frac{2 l + 1}{4 \pi}} = \frac{c + 1}{\sqrt{4 \pi}} .
\end{equation*}
Based on \cref{eqn:form_ei}, we can obtain the final upper bound by using the Cauchy-Schwarz 
and triangle inequalities as follows,
\begin{align}
  | e_i (y) |
  & \leqslant \left| \frac{4 \pi}{|Y_p|} e_i (Y_p)^T M \Phi(y)
    \right| + \left| \frac{4 \pi}{|Y_p|} R_c(Y_p)^T M \Phi(y) \right| + |R_c(y)| \nonumber\\
  & \leqslant \frac{4 \pi}{|Y_p|} \| e_i (Y_p) \|_2 \| M \Phi (y) \|_2 +
    \frac{4 \pi}{|Y_p|} \| R_c (Y_p) \|_2 \| M \Phi (y) \|_2 + | R_c (y) |  \nonumber\\
%  & \leqslant \frac{4 \pi}{|Y_p|} \| e_i (Y_p)^TM \|_2 \|\Phi (y) \|_2 +
%  \frac{4 \pi}{|Y_p|} \| R_c (Y_p)^TM \|_2 \| \Phi (y) \|_2 + | R_c (y) |  \nonumber\\  
  & \leqslant (c + 1) \frac{\| e_i (Y_p) \|_2}{\sqrt{|Y_p|}} + (c + 1)
  \frac{\| R_c (Y_p) \|_2}{\sqrt{|Y_p|}} + | R_c (y) | \nonumber\\
  & \leqslant (c + 1) \frac{\| e_i (Y_p) \|_2}{\sqrt{|Y_p|}} + (c + 2) \frac{(1 +
  | X_\text{rep} | \| u_i \|_{\infty})}{r_2 - r_1} \left( \frac{r_1}{r_2}
  \right)^{c + 1} .\label{eqn:bound2}
\end{align}

\end{proof}

Combining \cref{prop} and inequality \cref{eqn:maxbound}, the error bound 
of the proxy surface method for the ID approximation of $K(X_0, Y_0)$ is described 
as follows. 
\begin{theorem}[Error bound for the proxy surface method]\label{theorem}
If point set $Y_p$ satisfies the condition in \cref{prop} and the ID \cref{eqn:proxy} 
of $K(X_0, Y_p)$ has precision $\varepsilon\sqrt{|Y_p|}$, i.e., 
$\|e_i(Y_p)\|_2 \leqslant \varepsilon\sqrt{|Y_p|}$ for each $x_i \in X_0$, 
the ID \cref{eqn:target} of $K(X_0, Y_0)$ in the proxy surface method has 
error $e_i(Y_0)$ at the $i$th row bounded as
%	\begin{equation}\label{eqn:boundvector}
%	\|e_i(Y_0)\|_2 \leqslant (c + 1) \sqrt{|Y_0|}\varepsilon + 
%	(c + 2) \sqrt{|Y_0|} \frac{(1 + | X_\textnormal{rep} | \| u_i \|_{\infty})}{r_2 - r_1}
%	\left( \frac{r_1}{r_2} \right)^{c + 1}. 
%	\end{equation}
	\begin{align}
	\dfrac{\|e_i(Y_0)\|_2}{\sqrt{|Y_0|}} 
	&\leqslant (c+1) \dfrac{\|e_i(Y_p)\|_2}{\sqrt{|Y_p|}} + (c+2) \frac{(1+|X_\textnormal{rep}| \|u_i \|_{\infty})}{r_2-r_1} \left( \frac{r_1}{r_2} \right)^{c + 1} \label{eqn:boundvector1} \\
	&\leqslant (c+1) \varepsilon + (c+2) \frac{(1+|X_\textnormal{rep}| \|u_i \|_{\infty})}{r_2-r_1} \left( \frac{r_1}{r_2} \right)^{c + 1}. \label{eqn:boundvector}
	\end{align}
\end{theorem}

When there are not enough points in $Y_p$, i.e., $c$ is small, the error is dominated by 
the second term of the upper bound in \cref{eqn:boundvector} which comes from the truncation error $R_c(y)$
in \cref{eqn:multipole_ei}.
A simple interpretation is that controlling the values of $e_i(y)$ for $y \in Y_p$ through 
the ID approximation of $K(X_0, Y_p)$ is not sufficient to completely control $e_i(y)$ over the whole 
surface $\Gamma$. 

%On the other hand, when $c$ is sufficiently large, $R_c(y)$ exponentially decays to zero and the first term in 
%\cref{eqn:boundvector} dominates the error.
%Since $e_i(y)$ is always a linear combination of $\{K(x_j, y)\}_{x_j \in X_0}$ with bounded combination coefficients, 
%the coefficients $\{E_c^m\}_{-c\leqslant m \leqslant c}$ and the truncation error $R_c(y)$ in \cref{eqn:multipole_ei} both
% decay exponentially to zero with increasing $c$. 
%For a large $c$ such that $R_c(y)$ and $\{E_l^m\}_{l > c}$ are of machine precision scale, 
%$\frac{4\pi}{|Y_p|}e_i(Y_p)^TM$ accurately approximates the vector $(E_l^m)_{l,m}$ with $l$ up to $c$ as discussed in \cref{eqn:coef} 
%and \cref{eqn:form_ei}. 
%Thus, the Cauchy-Schwarz inequality $|e_i(Y_p)^T M \Phi(y)| \leqslant \|e_i(Y_p)^T M\|_2 \|\Phi(y)\|_2$ used in \cref{eqn:bound2} will
%not be tight if further increase $c$ since the tail entries of $\frac{4\pi}{|Y_p|}e_i(Y_p)^T M$ will become close to zero. 
%Furthermore, it means that when $c$ is sufficiently large, the first term in the upper bound \cref{eqn:boundvector} can be replaced by 
%a constant that is of scale $O(\sqrt{|Y_0|}\varepsilon)$ but unrelated to $c$. 

\section{Selection of $Y_p$}\label{sec:Yp}
Using the quadrature point sets provided in \cite{womersley2017efficient}, only $2c^2 + 2c + O(1)$ points are 
needed in $Y_p$ to make the associated numerical quadrature exact for polynomials on $\Gamma$ of degree up to $2c$.
Thus, the key for the selection of $Y_p$ is to decide the smallest $c$ for a given error threshold to balance 
the precision and efficiency of the proxy surface method.

Since the error bound \cref{eqn:boundvector} contains $|X_\text{rep}|$ and $\| u_i \|_{\infty}$
that depend on the ID of $K(X_0, Y_p)$, we need some a priori bounds of these two 
quantities for the selection of $Y_p$. When using sRRQR to find the ID of $K (X_0, Y_p)$,
entries of $U$ can be bounded by a pre-specified parameter $C_\text{qr} \geqslant 1$ and thus 
$\| u_i \|_{\infty} \leqslant C_\text{qr}$ for any $x_i \in X_0$. 
$|X_\text{rep}|$ is a rank estimate of $K(X_0, Y_p)$ and thus satisfies
$|X_\text{rep}| \leqslant \min(|X_0|, |Y_p|)$.
Plugging these values into \cref{eqn:boundvector}, we obtain an a priori error bound as
\[
%  \| e_i (Y_0) \|_2 \leqslant (c + 1) \sqrt{|Y_0|}\varepsilon + (c + 2) \sqrt{|Y_0|}
%  \frac{C_\text{qr}\min\left(|X_0|, |Y_p|\right) + 1}{r_2 - r_1} \left(\frac{r_1}{r_2}\right)^{c+1}.
  \dfrac{\| e_i (Y_0) \|_2}{\sqrt{|Y_0|}} \leqslant (c + 1) \varepsilon + (c + 2)
  \frac{C_\text{qr}\min\left(|X_0|, |Y_p|\right) + 1}{r_2 - r_1} \left(\frac{r_1}{r_2}\right)^{c+1}.
\]

Heuristically, we choose the integer constant $c$ by making the second term above of scale $(c+1)\varepsilon$, i.e.,
\begin{equation}\label{eqn:decide_c}
\frac{C_\text{qr}\min\left(|X_0|, 2c^2+2c+O(1)\right) + 1}{r_2 - r_1} \left(\frac{r_1}{r_2}\right)^{c+1}
\approx \varepsilon.
\end{equation}
$Y_p$ can then be directly obtained from the dataset of \cite{womersley2017efficient} with the selected $c$.
The row approximation error of the proxy surface method is then bounded as 
\begin{equation}\label{eqn:bound3}
\|e_i(Y_0)\|_2 \leqslant (2c+3)\varepsilon\sqrt{|Y_0|}.
\end{equation}

It is worth noting that the condition for $Y_p$ in \cref{prop} is mainly for a 
rigorous analysis and also that the obtained upper bounds \cref{eqn:boundvector} and
\cref{eqn:bound3} may not be tight. 
Thus, the above selection of $Y_p$ is conservative and may be unnecessarily large. 
However, the key idea conveyed by \cref{theorem} and the above selection of $Y_p$ is that 
as long as $\Gamma$ and $\mathcal{X}$ are well-separated, e.g., $r_2 - r_1 \geqslant 1$, and 
the ratio of their radii is fixed, $Y_p$ with a constant number of points is sufficient to
maintain the accuracy of the proxy surface method. 
Also, this theorem rigorously justifies the claim in \cite{ho_fast_2012} about using a 
constant number of points to discretize different proxy surfaces in recursive skeletonization.

%%%%%%%%%%%%%%%%%%%%%%%%%%%%%%%%%%%%%%%%%%%%%%%%%%%%%%%%%%%%%%%%%%%%%%%%%%%%%%%%%%%%%%%%%
%%%%%%%%%%%%%%%%%%%%%%%%%%%%%%%%%%%%%%%%%%%%%%%%%%%%%%%%%%%%%%%%%%%%%%%%%%%%%%%%%%%%%%%%%
%%%%%%%%%%%%%%%%%%%%%%%%%%%%%%%%%%%%%%%%%%%%%%%%%%%%%%%%%%%%%%%%%%%%%%%%%%%%%%%%%%%%%%%%%

\section{Numerical experiments}

We consider the 3D Laplace kernel $K (x, y) = 1 / | x - y |$. The error threshold for the ID 
approximation of $K (X_0, Y_p)$ is set as $\varepsilon\sqrt{| Y_p |}$ so that 
$\| e_i (Y_p) \|_2  \leqslant \varepsilon \sqrt{| Y_p |} $ for each $x_i \in X_0$ 
with $\varepsilon$ specified later. 
The entry-bound parameter $C_\text{qr}$ for sRRQR in the ID approximation of $K(X_0, Y_p)$ 
is set to $2$. 
%Unless otherwise specified, in the following tests, we consider the 
%domain pair $\mathcal{X} = B(0, 1)$ and $\mathcal{Y} = \mathbb{R}^3\backslash B(0, 2)$ and error
%threshold $\varepsilon = 10^{-6}$. The corresponding constant $c$ estimated by 
%\cref{eqn:decide_c} is 30 and the selected $Y_p$ from \cite{womersley2017efficient} has 1862
%points.

\begin{figure}[t]
	\centering
	\subfloat[]{
		\includegraphics[width=0.49\textwidth]{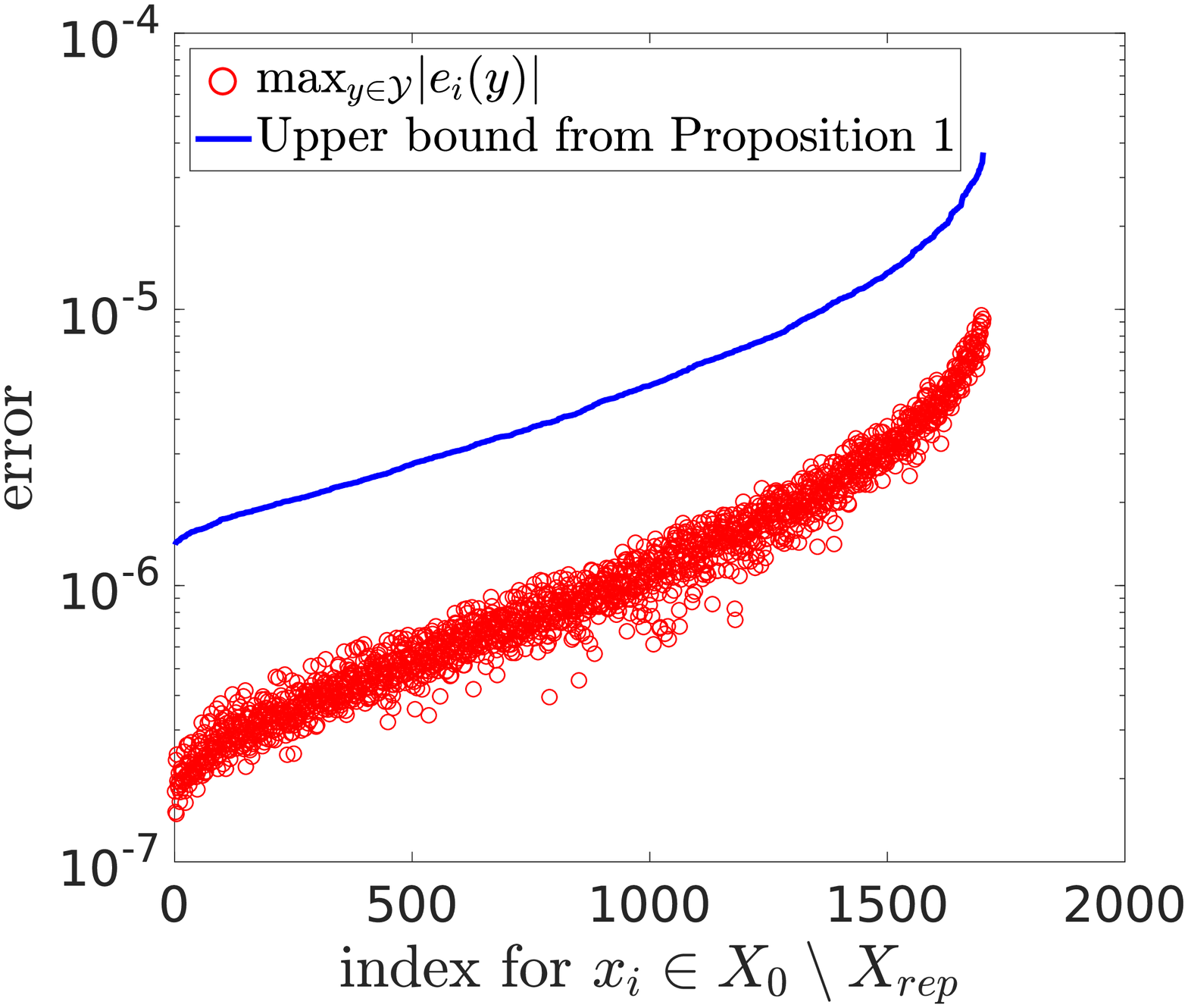}
	}
	\subfloat[]{
		\includegraphics[width=0.49\textwidth]{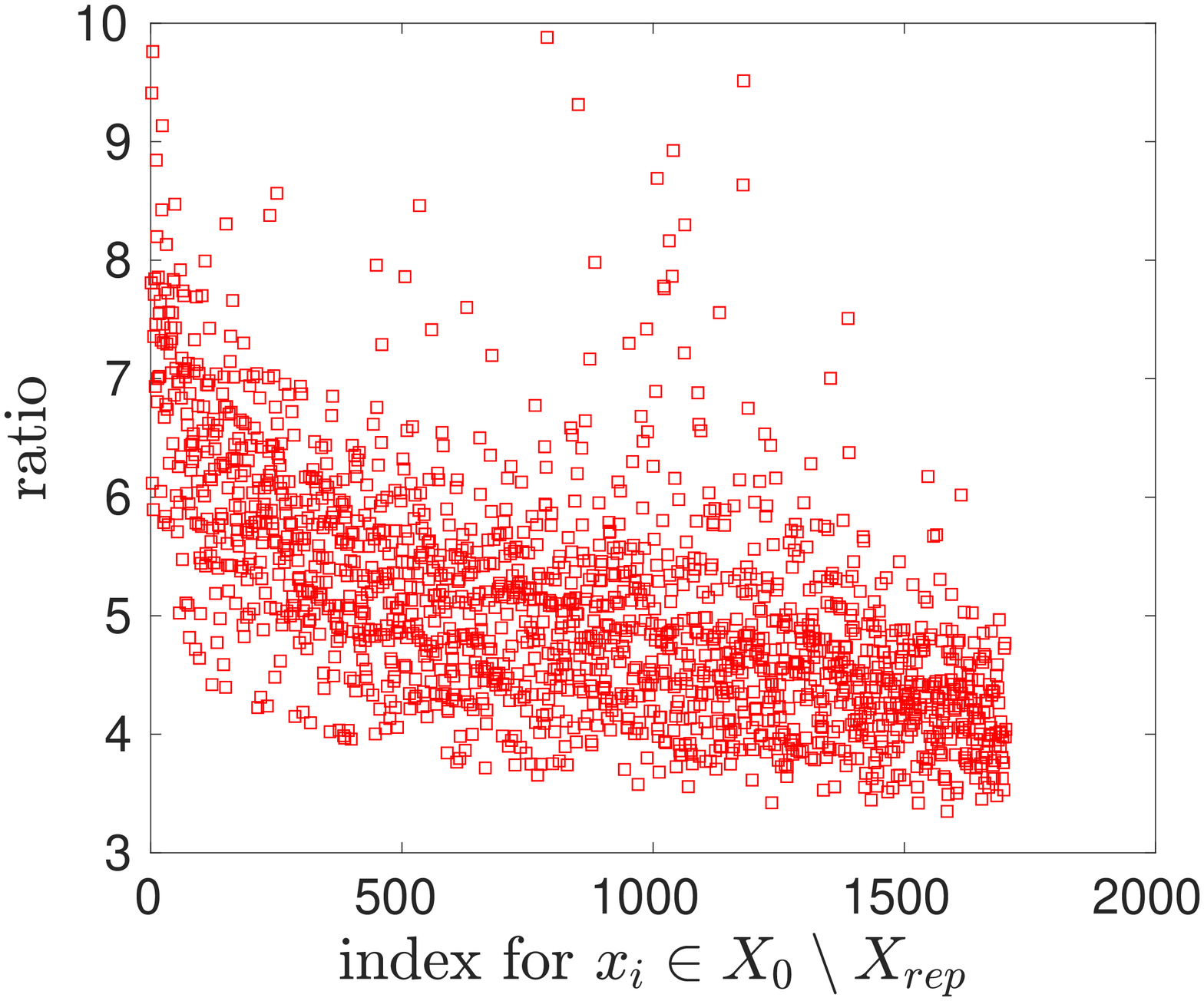}	
	}
	\caption{
		Values of $\max_{y \in \mathcal{Y}} |e_i(y)|$ and its upper bound given in 
		\cref{eqn:bound} for each $x_i \in X_0 \backslash X_\text{rep}$ with (a) values of the 
		two quantities and (b) ratio of the upper bound to $\max_{y \in \mathcal{Y}} |e_i(y)|$. 
		Indices for $x_i \in X_0 \backslash X_\text{rep}$ are sorted so that the upper bounds
		are in ascending order.\label{fig:test1}	
	}
\end{figure}

\subsection{Error bound for $e_i(y)$ in \cref{prop}}
Consider domain pair $\mathcal{X}\times \mathcal{Y} = B(0,1) \times (\mathbb{R}^3\backslash B(0, 2))$
and error threshold $\varepsilon = 10^{-6}$. 
The corresponding constant $c$ estimated by \cref{eqn:decide_c} is 30 and $Y_p$ selected from
\cite{womersley2017efficient} has 1862 points.
We randomly and uniformly select 2000 points in $\mathcal{X}$ for $X_0$. 
The ID of $K(X_0, Y_p)$ obtains $X_\text{rep}$ with 298 points and also 
defines $e_i(y)$ for each $x_i \in X_0$.  

To check the error bound \cref{eqn:bound} in \cref{prop}, we plot $\max_{y \in \mathcal{Y}} |e_i(y)|$ 
and its bound in \cref{fig:test1} for each $x_i \in X_0 \backslash X_\text{rep}$
\footnote{For any $x_i \in X_\text{rep}$, $e_i(y)$ is the zero function.}\
where, according to \cref{eqn:maximum_principle}, 
$\max_{y \in \mathcal{Y}} |e_i(y)|$ is estimated by densely sampling $|e_i(y)|$
over $\Gamma$. 
As can be observed, the upper bound in \cref{prop} is usually within an order of magnitude of
$\max_{y \in \mathcal{Y}} |e_i(y)|$ for each $e_i(y)$. 
However, the ratio of these two quantities being always larger than 3 indicates that an even 
sharper upper bound may exist. 
%\footnotetext{For any $x_i \in X_\text{rep}$, $e_i(y)$ is the zero function.}

In a further numerical test, we vary the constant $c$ and thus the corresponding $Y_p$ selected 
from \cite{womersley2017efficient}. For each set of $e_i(y)$ obtained from different $Y_p$, 
in \cref{fig:test2}, we plot $\max_{x_i\in X_0, y\in \mathcal{Y}} |e_i(y)|$ and 
its upper bound derived from \cref{prop}, i.e., 
\begin{equation}\label{eqn:general_bound}
	\max_{x_i\in X_0, y\in \mathcal{Y}} |e_i(y)| 
	\leqslant 
	(c+1)\varepsilon + (c+2) \dfrac{1+\max_i\|u_i\|_\infty|X_\text{rep}|}{r_2-r_1}\left(\dfrac{r_1}{r_2}\right)^{c+1}.
\end{equation}

\begin{figure}[h]
	\centering
	\includegraphics[width=0.70\textwidth]{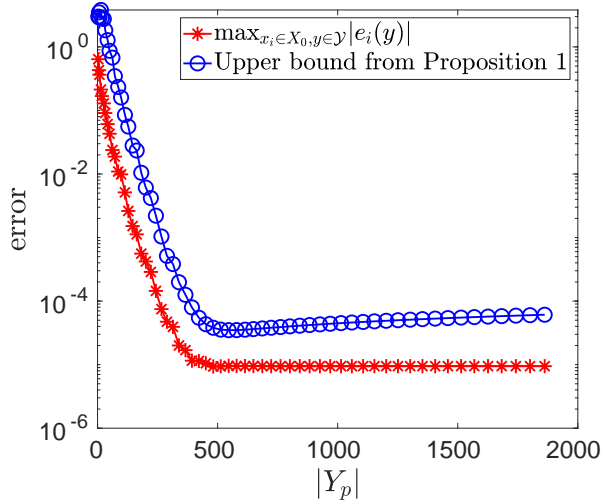}
	\caption{
		Values of $\max_{x_i \in X_0, y \in \mathcal{Y}} |e_i(y)|$ and its upper bound 
		\cref{eqn:general_bound} for different constants $c$ and corresponding different point
		sets $Y_p$ selected from \cite{womersley2017efficient}. 
		The error threshold $\varepsilon \sqrt{|Y_p|}$ with fixed $\varepsilon = 10^{-6}$ is used in the ID approximation of $K(X_0, Y_p)$ for different $Y_p$. 
	}
	\label{fig:test2}
\end{figure}

From the numerical results, the upper bound \cref{eqn:general_bound} is quite tight and it also catches the knee at
$|Y_p|\approx 500$ where $\max_{x_i\in X_0, y\in \mathcal{Y}} |e_i(y)|$ stops decreasing. 
Note that $\max_{x_i\in X_0, y\in \mathcal{Y}} |e_i(y)|$ not further decreasing with larger $|Y_p|$ 
is due to the error threshold $\varepsilon\sqrt{|Y_p|}$ used in the ID approximation of $K(X_0, Y_p)$. 
The knee also shows that approximately 500 points for $Y_p$ should be enough to obtain the lowest 
error for the proxy surface method in this problem setting. 
However, the method of choosing $Y_p$ introduced in \cref{sec:Yp} gives $c = 30$ and $|Y_p| = 1862$.
The main cause of this overestimation of $|Y_p|$, by comparing \cref{eqn:general_bound} and 
\cref{eqn:decide_c}, turns out to be the looseness of 
$|X_\text{rep}| \leqslant \min(|X_0|, |Y_p|)$ utilized in \cref{eqn:decide_c}.

\subsection{Error bound for $\|e_i(Y_0)\|_2$ in \cref{theorem}}
The bound for $\|e_i(Y_0)\|_2$ in \cref{theorem} simply combines the bound for 
$\max_{y \in \mathcal{Y}} |e_i(y)|$ in \cref{prop}, which has been shown in the previous test 
to be quite tight, and the inequality \cref{eqn:maxbound}, i.e., 
$\|e_i(Y_0)\|_2 \leqslant \sqrt{|Y_0|} \max_{y \in \mathcal{Y}} |e_i(y)|$. 
Note that equality of \cref{eqn:maxbound} can hold when $|e_i(y)|$ reaches its maximum 
at all the points in $Y_0$. However, for $Y_0$ with an arbitrary point distribution, this 
inequality turns out to be quite loose as illustrated below.

We use the same $\mathcal{X}\times \mathcal{Y}$, $\varepsilon$ and $X_0\subset \mathcal{X}$ as 
in the previous test and consider the point set $Y_p$ associated with $c=30$. 
We randomly and uniformly select 20000 points for $Y_0$ in two subdomains of $\mathcal{Y}$, 
$B(0,4)\backslash B(0,2)$ and $B(0, 8)\backslash B(0,2)$.
With the proxy surface method, the obtained average entry-wise error $\|e_i(Y_0)\|_2/\sqrt{|Y_0|}$ and 
the maximum entry error $\max_{y \in Y_0}|e_i(y)|$ for each $x_i \in X_0 \backslash X_\text{rep}$ are 
plotted in \cref{fig:test3} along with their shared upper bound \cref{eqn:boundvector1} given in \cref{theorem}. 

\begin{figure}[ht]
	\centering
	\subfloat[$Y_0 \subset B(0,4)\backslash B(0,2)$]{
		\includegraphics[width=0.49\textwidth]{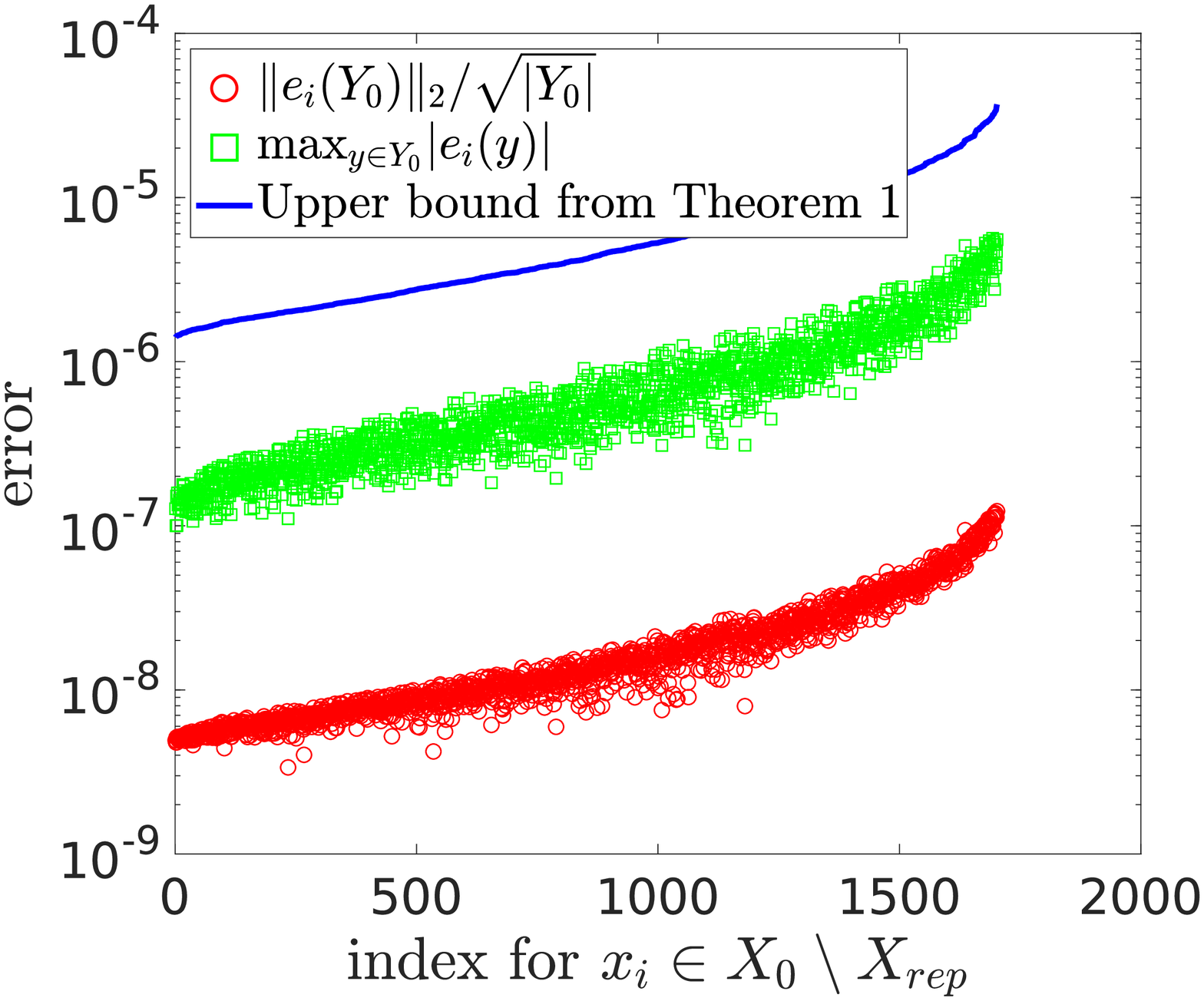}
	}
	\subfloat[$Y_0 \subset B(0,8)\backslash B(0,2)$]{
		\includegraphics[width=0.49\textwidth]{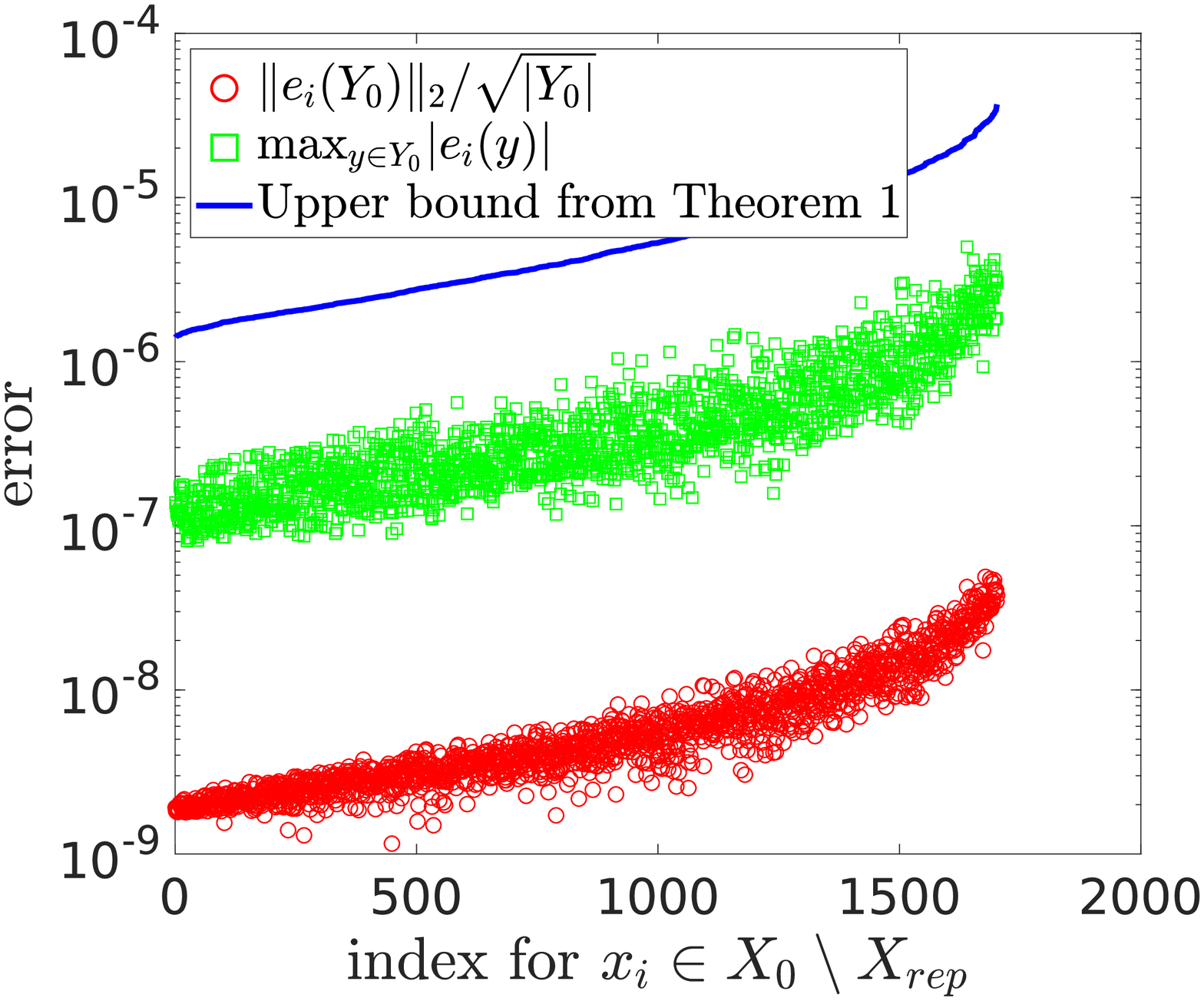}	
	}
	\caption{
		Values of $\|e_i(Y_0)\|_2/\sqrt{|Y_0|}$, $\max_{y \in Y_0} |e_i(y)|$ and their shared upper bound 
		\cref{eqn:boundvector1} given in \cref{theorem} with different point distributions in $Y_0$.
		Indices for $x_i \in X_0 \backslash X_\text{rep}$ are sorted so that the upper bounds
		are in ascending order.\label{fig:test3}	
	}
\end{figure}

For both choices of the subdomain of $\mathcal{Y}$ (and $Y_0$), 
$\|e_i(Y_0)\|_2/\sqrt{|Y_0|}$ is more than one order of magnitude smaller than 
$\max_{y \in Y_0} |e_i(y)|$. 
%is close to $\max_{y \in \mathcal{Y}} |e_i(y)|$ obtained in \cref{fig:test1} while 
Thus, the inequality \cref{eqn:maxbound} is quite loose in these cases. 
%However, no further improvement can be made over this basic estimation since the proxy surface method 
%has no constraint on the point distribution of $Y_0$ except that $Y_0$ has to be in $\mathcal{Y}$. 

\subsection{Selection of $Y_p$}
From \cref{sec:Yp}, the selection of $Y_p$ mainly depends on the domain pair 
$\mathcal{X} \times \mathcal{Y}$ and the ID error threshold $\varepsilon\sqrt{|Y_p|}$. 
Varying these parameters, \cref{tab:setting} lists the number of points in the selected $Y_p$. 
Although our selection scheme is quite conservative as shown in \cref{fig:test2}, 
the results in \cref{tab:setting} clearly show how the selection of $Y_p$ is affected by these parameters.

\begin{table}[H]
	\centering
	\caption{
		Estimated constant $c$ and number of points in selected $Y_p$ under different settings 
		of the radii $r_1$ and $r_2$ for the domain pair 
		$\mathcal{X}\times\mathcal{Y} = B(0, r_1)\times \left(\mathbb{R}^3 \backslash B(0, r_2)\right)$ and 
		$\varepsilon$ in the ID error threshold $\varepsilon\sqrt{|Y_p|}$. 
		\label{tab:setting}
	}
	
	\begin{tabular}{ccccccc}
		\toprule
		& $r_1$ & $r_2$ & $r_1/r_2$ & $\varepsilon$ & $c$ & $|Y_p|$ \\
		\midrule
		reference test & 1  & 2  & 0.5  & $10^{-6}$ & 30 & 1862 \\
		\multirow{2}{*}{different $\varepsilon$} 
					   & 1  & 2  & 0.5  & $10^{-4}$ & 23 & 1106 \\
					   & 1  & 2  & 0.5  & $10^{-8}$ & 38 & 2965 \\
	    \multirow{2}{*}{different $\dfrac{r_1}{r_2}$} 
					   & 1  & 4  & 0.25  & $10^{-6}$ & 12 & 314 \\
					   & 1  & 6  & 0.16  & $10^{-6}$ & 9  & 181 \\
        \multirow{2}{*}{different $r_2-r_1$} 
					   & 10  & 20   & 0.5  & $10^{-6}$ & 27 & 1514 \\
					   & 100 & 200  & 0.5  & $10^{-6}$ & 23 & 1106 \\
		\bottomrule
	\end{tabular}
\end{table}

\section{Conclusion}
The error analysis in this paper rigorously confirms the accuracy of the proxy surface method by showing the
quantitative relationship \cref{eqn:boundvector1} between the error of the ID of $K(X_0, Y_0)$ and 
the error of the ID of $K(X_0, Y_p)$.
%The error analysis in this paper shows that the ID of $K(X_0, Y_0)$ defined in the proxy surface method
%has average entry-wise error at each row, i.e., $\|e_i(Y_0)\|_2/\sqrt{|Y_0|}$, of the same scale as the ID of $K(X_0, Y_p)$. 
%Thus, the proxy surface method can indeed achieve good accuracy.
%
%The error analysis in this paper shows the good accuracy of the proxy surface method by estimating
%the quantitative relation between the error of the ID of $K(X_0, Y_p)$ and that of the defined ID of $K(X_0, Y_0)$. 
Also, the analysis justifies the use of a constant number of points to discretize proxy surfaces of different sizes
in the hierarchical matrix construction of 3D Laplace kernel matrices, when the ratio $r_1/r_2$
is constant. 
The same error analysis technique can also be applied to the proxy surface method for more general matrices with entries 
defined by the interactions between two compact charge distributions, e.g., the matrix in the Galerkin method for integral
equations and the electron repulsion integral tensors with Gaussian-type basis functions.

%\nolinenumbers
\bibliographystyle{elsarticle-num}
\bibliography{references}

\end{document}